\documentclass[a4paper]{article}
\usepackage{fullpage}
\usepackage{amsmath}
\usepackage{amssymb}
\usepackage{amsthm}
\usepackage{graphicx}
\usepackage{bbm}
\usepackage{enumerate}
\usepackage{microtype}
\usepackage{tocbibind}
\usepackage{xcolor}
\usepackage{indentfirst}
\usepackage[none]{hyphenat}
\usepackage[colorlinks=true,urlcolor=blue,linkcolor=blue,plainpages=false,pdfpagelabels]{hyperref}
\allowdisplaybreaks

\newtheorem{theorem}{Theorem}[section]
\newtheorem{proposition}[theorem]{Proposition}

\newtheorem{corollary}[theorem]{Corollary}

\newtheorem{definition}[theorem]{Definition}

\newcommand{\N}{\mathbb{N}}

\newcommand{\R}{\mathbb{R}}

\newcommand{\eps}{\varepsilon}

\newcommand{\vv}[1]{\overrightarrow{#1}}

\newcommand{\piv}[2]{\pi(\vv{#1},\vv{#2})}
\newcommand{\piiv}[3]{\pi_{#1}(\vv{#2},\vv{#3})}

\newcommand{\Hh}{\mathbb{H}}
\DeclareMathOperator{\arcosh}{arcosh}
\title{Products of hyperbolic spaces}
\author{Pedro Pinto${}^{a}$ and Andrei Sipo\c s${}^{b,c,d}$\\[2mm]
\footnotesize ${}^a$Department of Mathematics, Technische Universit\"at Darmstadt,\\
\footnotesize Schlossgartenstra\ss{}e 7, 64289 Darmstadt, Germany\\[1mm]
\footnotesize ${}^b$Research Center for Logic, Optimization and Security (LOS), Department of Computer Science,\\
\footnotesize Faculty of Mathematics and Computer Science, University of Bucharest,\\
\footnotesize Academiei 14, 010014 Bucharest, Romania\\[1mm]
\footnotesize ${}^c$Simion Stoilow Institute of Mathematics of the Romanian Academy,\\
\footnotesize Calea Grivi\c tei 21, 010702 Bucharest, Romania\\[1mm]
\footnotesize ${}^d$Institute for Logic and Data Science,\\
\footnotesize Popa Tatu 18, 010805 Bucharest, Romania\\[2mm]
\footnotesize Emails: pinto@mathematik.tu-darmstadt.de, andrei.sipos@fmi.unibuc.ro\\
}
\date{}

\begin{document}

\maketitle

\begin{abstract}
The class of uniformly smooth hyperbolic spaces was recently introduced by the first author as a common generalization of both CAT(0) spaces and uniformly smooth Banach spaces, in a way that Reich's theorem on resolvent convergence could still be proven. We define products of such spaces, showing that they are reasonably well-behaved. In this manner, we provide the first example of a space for which Reich's theorem holds and which is neither a CAT(0) space, nor a convex subset of a normed space.


\end{abstract}

\section{Introduction}

In the last few decades, there has been a renewed interest in nonlinear (metric) generalizations of the classical structures of functional analysis, primarily motivated by fixed point theory and convex optimization. The center of these developments has been the class of {\it CAT(0) spaces}, which are metric spaces that aim to model the idea of non-positive curvature (directly generalizing inner product spaces). More general spaces of a `hyperbolic' nature have also been studied, and a particularly flexible notion has been introduced by Kohlenbach \cite{Koh05} under the name of {\it $W$-hyperbolic spaces} (see the next section for a definition and \cite[pp. 384--388]{Koh08} for a detailed discussion on the relationship between various definitions of hyperbolicity).

Recently, the first author \cite{Pin24} has defined a subclass of $W$-hyperbolic spaces which he called {\it uniformly smooth hyperbolic spaces}. In order to understand the motivation for this concept, in addition to generalizing the Banach space notion of uniform smoothness, we shall step back a bit to look at the context.

If $X$ is a metric space possessing enough convexity structure, and $T:X \to X$ is a nonexpansive mapping, one can define, for any $x \in X$, the point $x_t$, called a `resolvent', by the implicit definition $x_t=tTx_t+(1-t)x$ (existence and uniqueness being guaranteed by the Banach contraction principle). One may then ask about the strong convergence of the path $(x_t)$ (for $t \to 1$), and, for $X$ being a bounded convex subset of a Hilbert space, this was first proven in 1967 by Browder \cite{Bro67A} (and, in the same year, a more elementary proof was given by Halpern \cite{Hal67}). A little more than a decade later, the corresponding result for uniformly smooth Banach spaces, requiring significantly new ideas, was proven by Reich \cite{Rei80}.

In the past twenty years, these kinds of convergence results have begun to be studied from the point of view of {\it proof mining}, which is a research paradigm (first suggested by Kreisel in the 1950s under the name of `unwinding of proofs' and then highly developed starting in the 1990s by Kohlenbach and his students and collaborators) that aims to analyze ordinary mathematical proofs using tools from proof theory (a branch of mathematical logic) in order to extract additional -- typically quantitative -- information out of them (for more information, see Kohlenbach's monograph \cite{Koh08} and his recent research survey \cite{Koh19}).

The analysis of Browder's result was obtained by Kohlenbach in 2011 \cite{Koh11}; its sequel, the work on Reich's theorem, had to wait, like the original theorem itself, around ten years before it was finally achieved by Kohlenbach and the second author \cite{KohSip21}, being one of the most intricate results ever produced by the program of proof mining (see also the second author's recent extension \cite{Sip24} to Reich's original setting of accretive operators). The goals of these endeavours were to obtain quantitative results, namely so-called `rates of metastability'; what is more relevant here is that a proof-theoretic analysis usually strips down a proof to its bare essence, leaving it more suitable to generalization.

Such a generalization was needed, as per the aforementioned paper \cite{Pin24}, since the theorem had also been proven, in the meantime, in the setting of complete CAT(0) spaces (also known as Hadamard spaces) by Saejung \cite{Sae10} (and then proof-theoretically analyzed by Kohlenbach and Leu\c stean \cite{KohLeu12A}), and, thus, one had to have a framework to encompass both CAT(0) spaces and uniformly smooth Banach spaces, this being achieved by the class of uniformly smooth hyperbolic spaces. One additional issue was that Kohlenbach and the second author, in order to tame the proof of Reich's theorem and make it amenable to analysis and to generalization, introduced the additional uniform convexity hypothesis on the Banach space (thus, still maintaining $L^p$ spaces as a particular case of the theorem) -- fortunately, this notion had already been formulated in the hyperbolic setting by Leu\c stean \cite{Leu07,Leu10} in the form of $UCW$-hyperbolic spaces. Thus, the first author showed that Reich's theorem and its consequences do hold for complete uniformly smooth $UCW$-hyperbolic spaces. The generalization of the proof-theoretic analysis itself was ultimately relegated to the forthcoming paper \cite{PinXX}.

Now, the introduction of this new abstract notion raises the question of whether it solely provides a unified proof of the theorem for the two known cases or whether one has properly extended the class of spaces for which the theorem holds. This is the question that this paper aims to answer. 

Towards that end, we take the existing construction of a product of metric spaces and show that it extends all the way to uniformly smooth hyperbolic spaces. We also show that uniform convexity is preserved under products, by generalizing ideas due to Clarkson \cite{Cla36}. Finally, we give an example of a product space for which we show that it cannot be either a CAT(0) space or a convex subset of a normed space.

Section~\ref{sec:ns} recalls the main facts about the classes of spaces under consideration. Section~\ref{sec:main} presents all the results on the product construction, showing various preservation properties. Section~\ref{sec:xmpl} exhibits the concrete example mentioned above.

The symbol $\|\cdot\|$, unless specified otherwise, will denote the usual Euclidean norm.

\section{Nonlinear spaces}\label{sec:ns}

One says that a metric space $(X,d)$ is {\it geodesic} if for any two points $x$, $y \in X$ there is a {\it geodesic} that joins them, i.e.\ a mapping $\gamma : [0,1] \to X$ such that $\gamma(0)=x$, $\gamma(1)=y$ and for any $t$, $t' \in [0,1]$ we have that
$$d(\gamma(t),\gamma(t')) = |t-t'| d(x,y).$$

Generally, geodesics need not be unique, and, thus, one is led to the following definition, which considers geodesics on a metric space as an additional structure\footnote{The convexity function $W$ was first considered by Takahashi in \cite{Tak70} where a triple $(X,d,W)$ satisfying $(W1)$ is called a convex metric space. The notion used here, frequently considered nowadays to be the nonlinear generalization of convexity in normed spaces, was introduced by Kohlenbach in \cite{Koh05}. As said in the Introduction, see \cite[pp. 384--388]{Koh08} for a detailed discussion on the relationship between various definitions of hyperbolicity and on the proof-theoretical considerations that ultimately led to the adoption of this one. We only mention here that this notion is more general than that of hyperbolic spaces in the sense of Reich and Shafrir \cite{ReiSha90}, and slightly more restrictive than the setting due to Goebel and Kirk \cite{GoeKir83} of spaces of hyperbolic type. } (and not as a property), which is required to satisfy additional properties.

\begin{definition}
A {\bf $W$-hyperbolic space} is a triple $(X,d,W)$ where $(X,d)$ is a metric space and $W: X^2 \times [0,1] \to X$ such that, for all $x$, $y$, $z$, $w \in X$ and $\lambda$, $\mu \in [0,1]$, we have that
\begin{enumerate}[(W1)]
\item $d(z,W(x,y,\lambda)) \leq (1-\lambda)d(z,x) + \lambda d(z,y)$;
\item $d(W(x,y,\lambda),W(x,y,\mu)) = |\lambda-\mu|d(x,y)$;
\item $W(x,y,\lambda)=W(y,x,1-\lambda)$;
\item $d(W(x,z,\lambda),W(y,w,\lambda))\leq(1-\lambda) d(x,y) + \lambda d(z,w)$.
\end{enumerate}
\end{definition}

Clearly, any normed space may be made into a $W$-hyperbolic space in a canonical way. 

A subset $C$ of a $W$-hyperbolic space $(X,d,W)$ is called {\it convex} if, for any $x$, $y \in C$ and $\lambda \in [0,1]$, $W(x,y,\lambda) \in C$. If $(X,d,W$) is a $W$-hyperbolic space, $x$, $y \in X$ and $\lambda \in [0,1]$, we denote the point $W(x,y,\lambda)$ by $(1-\lambda)x+\lambda y$. We will, also, mainly write $\frac{x+y}2$ for $\frac12 x + \frac12 y$.

\begin{proposition}\label{quad}
Let $(X,d,W)$ be a $W$-hyperbolic space and $x$, $y$, $a \in X$. Then
$$d^2\left(\frac{x+y}2,a\right) \leq \frac{d^2(x,a)+d^2(y,a)}2 \leq \max(d^2(x,a),d^2(y,a)).$$
\end{proposition}

\begin{proof}
Using $(W1)$, we get that
$$d^2\left(\frac{x+y}2,a\right) \leq \left(\frac{d(x,a)+d(y,a)}2\right)^2\leq \frac{d^2(x,a)+d^2(y,a)}2,$$
the last inequality being immediate.
\end{proof}

As per \cite{Koh05,Leu07}, a particular nonlinear class of $W$-hyperbolic spaces is the one of CAT(0) spaces, introduced by A. Aleksandrov \cite{Ale51} and named as such by M. Gromov \cite{Gro87}, which may be defined as those geodesic spaces $(X,d)$ such that, for any geodesic $\gamma : [0,1] \to X$ and for any $z \in X$ and $t \in [0,1]$ we have that
$$d^2(z,\gamma(t)) \leq (1-t)d^2(z,\gamma(0)) + td^2(z,\gamma(1)) - t(1-t)d^2(\gamma(0),\gamma(1)).$$
Another well-known fact about CAT(0) spaces is that each such space $(X,d)$ is {\it uniquely geodesic} -- that is, for any $x$, $y \in X$ there is a unique geodesic $\gamma : [0,1] \to X$ that joins them -- this property giving their unique $W$-hyperbolic space structure. By the discussion in \cite[pp. 387--388]{Koh08}, CAT(0) spaces may even be defined as those $W$-hyperbolic spaces $(X,d,W)$ such that, for any $a$, $x$, $y \in X$,
\begin{equation}\label{cat0}
d^2\left(\frac{x+y}2,a\right) \leq \frac12 d^2(x,a) + \frac12 d^2(y,a) - \frac14 d^2(x,y).
\end{equation}

In particular, any inner product space is a CAT(0) space, and, indeed, as experience shows, CAT(0) spaces may be regarded as the rightful nonlinear generalization of inner product spaces.

In 2008, Berg and Nikolaev proved (see \cite[Proposition 14]{BerNik08}) that in any metric space $(X,d)$, the function $\langle\cdot,\cdot\rangle : X^2 \times X^2 \to \mathbb{R}$, defined, for any $x$, $y$, $u$, $v \in X$, denoting an ordered pair of points $(a,b) \in X^2$ by $\vv{ab}$, by
$$\langle \vv{xy}, \vv{uv} \rangle := \frac12(d^2(x,v) + d^2(y,u) - d^2(x,u) -d^2(y,v)),$$
called the {\it quasi-linearization function}, is the unique one such that, for any $x$, $y$, $u$, $v$, $w \in X$, we have that:
\begin{enumerate}[(i)]
\item $\langle\vv{xy},\vv{xy}\rangle = d^2(x,y)$;
\item $\langle\vv{xy},\vv{uv}\rangle = \langle\vv{uv},\vv{xy}\rangle$;
\item $\langle\vv{yx},\vv{uv}\rangle = -\langle\vv{xy},\vv{uv}\rangle$;
\item $\langle\vv{xy},\vv{uv}\rangle + \langle\vv{xy},\vv{vw}\rangle = \langle\vv{xy},\vv{uw}\rangle$.
\end{enumerate}
The inner product notation is justified by the fact that if $X$ is a (real) inner product space, for any $x$, $y$, $u$, $v \in X$,
$$\langle\vv{xy},\vv{uv}\rangle = \langle x-y,u-v \rangle.$$
The main result of \cite{BerNik08}, namely Theorem 1 of said paper, characterized CAT(0) spaces as being exactly those geodesic spaces $(X,d)$ such that the metric generalization of Cauchy-Schwarz inequality is satisfied, i.e.\ for any $x$, $y$, $u$, $v \in X$,
$$\langle\vv{xy},\vv{uv}\rangle \leq d(x,y)d(u,v).$$

A representative example of a (complete) CAT(0) space is the {\it (hyperbolic) Poincar\'e upper half-plane model} (see also \cite{BriHae99}), having as the underlying set
$$\Hh:= \{(x_1,x_2) \in \R^2 \mid x_2 > 0\},$$
where, given the function $\arcosh : [1,\infty) \to [0,\infty)$, where for every $t \in [1,\infty)$, $\arcosh t = \ln (t+\sqrt{t^2-1})$, the distance function is defined as follows: for any $x=(x_1,x_2)$, $y=(y_1,y_2) \in \Hh$, one sets
\begin{equation}\label{dist-h}
d(x,y):= \arcosh \left( 1 + \frac{(y_1-x_1)^2 + (y_2-x_2)^2}{2x_2y_2} \right).
\end{equation}
It may be proven that geodesic lines of this space are of two types: for every $a\in \R$ and $r > 0$, one has the semicircle
$$\mathcal{C}_{a,r} = \{(x_1,x_2) \in \Hh \mid (x_1-a)^2+x_2^2=r^2\},$$
while, for every $a \in \R$, one has the ray
$$\mathcal{R}_a = \{(x_1,x_2) \in \Hh \mid x_1=a\}.$$
It may be then easily shown that for every two points there is exactly one geodesic segment that joins them. The general formula for the convex combination of two points is somewhat involved and we shall omit it, instead giving only the specialized formula for the midpoint, which shall be used later. For any $x=(x_1,x_2)$, $y=(y_1,y_2) \in \Hh$, we have that
\begin{equation}\label{mid-h}
W\left(x,y,\frac12\right) = \left(\frac{x_1y_2+x_2y_1}{x_2+y_2},\frac{\sqrt{x_2y_2}\cdot\sqrt{(x_2+y_2)^2+(x_1-y_1)^2}}{x_2+y_2}\right).
\end{equation}

We will now explore the class of uniformly smooth Banach spaces, in order to better understand the way it was generalized in \cite{Pin24} to uniformly smooth hyperbolic spaces.

\begin{definition}
Let $X$ be a Banach space. We define the {\bf normalized duality mapping of $X$} to be the map $J :X\to 2^{X^*}$, defined, for all $x \in X$, by
$$ J(x) := \{x^*\in X^* \mid  x^*(x)=\|x\|^2,\ \|x^*\|=\|x\| \}.$$
\end{definition}

A Banach space $X$ is called {\it smooth} if, for any $x \in X$ with $\|x\|=1$, we have that, for any $y \in X$ with $\|y\|=1$, the limit
\begin{equation}\label{limh}
\lim_{h \to 0} \frac{\|x+hy\|-\|x\|}{h}
\end{equation}
exists, and {\it uniformly smooth} if the limit is attained uniformly in $x$ and $y$. The condition of (not necessarily uniform) smoothness has been proven to be equivalent to the fact that the normalized duality mapping of the space, $J: X \to 2^{X^*}$, is single-valued -- and we shall denote its unique section by $j: X \to X^*$. Therefore, for all $x \in E$, $j(x)(x)=\|x\|^2$ and $\|j(x)\|=\|x\|$. Hilbert spaces are smooth, and clearly $j(x)(y)$ is then simply $\langle y,x \rangle$, for any $x,y$ in the space. Because of this, we may consider the $j$ to be a generalized variant of the inner product, sharing some of its nice properties, and, thus, one generally denotes, for all Banach spaces $X$, all $x^* \in X^*$ and $y \in X$, $x^*(y)$ by $\langle y,x^* \rangle$. In addition, the homogeneity of $j$ -- i.e.\ that, for all $x \in X$ and $t \in \R$, $j(tx)=tj(x)$ -- follows immediately from the definition of the duality mapping.

This view of $j$ as a generalization of the inner product, together with the way the quasi-linearization function was seen to be defined above, led the first author \cite{Pin24} to the following definition.

\begin{definition}[{\cite[Section 3]{Pin24}}]
A {\bf smooth hyperbolic space} is a quadruple $(X,d,W,\pi)$, where $(X,d,W)$ is a $W$-hyperbolic space and $\pi : X^2 \times X^2 \to \mathbb{R}$, such that, for any $x$, $y$, $u$, $v \in X$ (where, again, an ordered pair of points $(a,b) \in X^2$ is denoted by $\vv{ab}$):
\begin{enumerate}[(P1)]
\item $\piv{xy}{xy}=d^2(x,y)$;
\item $\piv{xy}{uv} = -\piv{yx}{uv} = -\piv{xy}{vu}$;
\item $\piv{xy}{uv} + \piv{yz}{uv} = \piv{xz}{uv}$;
\item $\piv{xy}{uv} \leq d(x,y)d(u,v)$;
\item $d^2(W(x,y,\lambda),z) \leq (1-\lambda)^2 d^2(x,z) + 2\lambda \piv{yz}{W(x,y,\lambda)z}$.
\end{enumerate}
\end{definition}

It is a classical result that, in uniformly smooth Banach spaces, the mapping $j$ is norm-to-norm uniformly continuous on bounded subsets. In the PhD thesis of B\'enilan \cite[p.\ 0.5, Proposition 0.3]{Ben72}, it is shown that the norm-to-norm uniform continuity on bounded subsets of an arbitrary duality selection mapping is in fact equivalent to uniform smoothness. A more recent proof which uses ideas due to Giles \cite{Gil67} may be found in \cite[Appendix A]{Kor15}. Motivated by this fact, the first author defined a {\it uniformly smooth hyperbolic space} to be a smooth hyperbolic space $(X,d,W,\pi)$ such that there is an $\omega : (0,\infty) \times (0,\infty) \to (0,\infty)$, called a {\it modulus of uniform continuity for $\pi$}, having the property that, for any $r$, $\eps>0$ and $a$, $u$, $v$, $x$, $y \in X$ with $d(u,a) \leq r$, $d(v,a) \leq r$ and $d(u,v) \leq \omega(r,\eps)$, one has that
$$| \piv{xy}{ua} - \piv{xy}{va}| \leq \eps \cdot d(x,y).$$

As mentioned in the Introduction, the remaining class of Banach spaces whose nonlinear generalization is relevant to our purposes is that of uniformly convex Banach spaces. We shall now turn to exploring uniform convexity directly in the hyperbolic setting, as first introduced by Leu\c stean \cite{Leu07,Leu10} (motivated by \cite[p. 105]{GoeRei84}; the monotonicity condition was justified by proof-theoretical considerations).

\begin{definition}
If $(X,d,W)$ is a $W$-hyperbolic space, then a {\bf modulus of uniform convexity} for $(X,d,W)$ is a function $\eta :(0, \infty) \times (0,2] \to (0,1]$ such that, for any $r >0$, any $\eps \in (0,2]$ and any $a$, $x$, $y \in X$ with $d(x,a) \leq r$, $d(y,a) \leq r$, $d(x,y) \geq \eps r$, we have that
$$d\left(\frac{x+y}2,a\right) \leq (1-\eta(r,\eps))r.$$
We call the modulus {\bf monotone} if, for any $r$, $s >0$ with $s \leq r$ and any $\eps \in (0,2]$, we have that $\eta(r,\eps) \leq \eta(s,\eps)$.

A {\bf $UCW$-hyperbolic space} is a $W$-hyperbolic space that admits a monotone modulus of uniform convexity.
\end{definition}

As remarked in \cite[Proposition 2.6]{Leu07}, CAT(0) spaces are $UCW$-hyperbolic spaces having as a modulus of uniform convexity the simple function $(r,\eps) \mapsto \eps^2/8$.

The crucial property of uniform convexity that was used in \cite{KohSip21} was the following, which we now reify for the first time.

\begin{definition}
Let $\psi :(0, \infty) \times (0,\infty) \to (0,\infty)$. We say that a $W$-hyperbolic space $(X,d,W)$ {\bf has property $(G)$ with modulus $\psi$} if, for any $r$, $\eps >0$ and any $a$, $x$, $y \in X$ with $d(x,a) \leq r$, $d(y,a) \leq r$, $d(x,y) \geq \eps $, we have that
$$d^2\left(\frac{x+y}2,a\right) \leq \frac12 d^2(x,a) + \frac12 d^2(y,a) - \psi(r,\eps).$$
\end{definition}

By \cite[Proposition 2.4]{KohSip21} (among others, see the references given there; this result essentially goes back to Z\u alinescu \cite[Section 4]{Zal83}), uniformly convex Banach spaces have property $(G)$ with a modulus which is easily computable in terms of the modulus of uniform convexity.

\begin{proposition}
CAT(0) spaces have property $(G)$.
\end{proposition}

\begin{proof}
Let $X$ be a CAT(0) space and $r$, $\eps >0$, $a$, $x$, $y \in X$ with $d(x,a) \leq r$, $d(y,a) \leq r$, $d(x,y) \geq \eps $.
By \eqref{cat0}, we have that
$$d^2\left(\frac{x+y}2,a\right) \leq \frac12 d^2(x,a) + \frac12 d^2(y,a) - \frac14 d^2(x,y) \leq \frac12 d^2(x,a) + \frac12 d^2(y,a) -\frac{\eps^2}4,$$
and we are done.
\end{proof}

We do not know (and we leave it as an open problem) whether, generally, $UCW$-hyperbolic spaces have property $(G)$; however, the first author identified in \cite[Proposition 5.4]{Pin24} (see also \cite[Proposition 2.5]{Sip21}) the following weaker property of them (which, again, we now reify for the first time) as being enough for the proof to go through.

\begin{definition}
Let $\psi :(0, \infty) \times (0,\infty) \to (0,\infty)$. We say that a $W$-hyperbolic space $(X,d,W)$ {\bf has property $(M)$ with modulus $\psi$} if, for any $r$, $\eps >0$ and any $a$, $x$, $y \in X$ with $d(x,a) \leq r$, $d(y,a) \leq r$, $d(x,y) \geq \eps $, we have that
$$d^2\left(\frac{x+y}2,a\right) \leq \max(d^2(x,a) ,d^2(y,a)) - \psi(r,\eps).$$
\end{definition}

Thus, the main results of \cite{Pin24} may be expressed by saying that Reich's theorem and its consequences hold for complete uniformly smooth hyperbolic spaces having property $(M)$, and, in particular, either having property $(G)$ or being $UCW$-hyperbolic.

\section{Main results}\label{sec:main}

Let $n \in \N \setminus \{0\}$ and fix $n$ metric spaces $(X_1,d_1),\ldots,(X_n,d_n)$. Put $X:=\prod_{i=1}^n X_i$ and define $d : X \times X \to \R$, by putting, for any $x=(x_1,\ldots,x_n)$, $y=(y_1,\ldots,y_n) \in X$,
$$d(x,y):=\left(\sum_{i=1}^n d_i^2(x_i,y_i) \right)^\frac12.$$
It is a classical result that $(X,d)$ is also a metric space (and it is complete if all of the factors are).

Now fix $W_1,\ldots,W_n$ such that, for each $i$, $(X_i,d_i,W_i)$ is a $W$-hyperbolic space. Define $W: X^2 \times [0,1] \to X$ by putting, for any $x=(x_1,\ldots,x_n)$, $y=(y_1,\ldots,y_n) \in X$ and $\lambda \in [0,1]$, and for any $i$,
$$W(x,y,\lambda)_i := W_i(x_i,y_i,\lambda).$$

\begin{proposition}
$(X,d,W)$ is a $W$-hyperbolic space.
\end{proposition}

\begin{proof}
We prove $(W1)$. Let $x=(x_1,\ldots,x_n)$, $y=(y_1,\ldots,y_n)$, $z=(z_1,\ldots,z_n) \in X$ and $\lambda \in [0,1]$. We have that
\begin{align*}
d(W(x,y,\lambda),z) &= \left(\sum_{i=1}^n d_i^2(W(x,y,\lambda)_i,z_i) \right)^\frac12 \\
&= \left(\sum_{i=1}^n d_i^2(W_i(x_i,y_i,\lambda),z_i) \right)^\frac12 \\
&\leq \left(\sum_{i=1}^n ((1-\lambda)d_i(x_i,z_i) + \lambda d_i(y_i,z_i))^2 \right)^\frac12 \\
&=\|((1-\lambda)d_1(x_1,z_1) + \lambda d_1(y_1,z_1),\ldots,(1-\lambda)d_n(x_n,z_n) + \lambda d_n(y_n,z_n))\| \\
&=\|(1-\lambda)(d_1(x_1,z_1),\ldots,d_n(x_n,z_n)) + \lambda(d_1(y_1,z_1),\ldots,d_n(y_n,z_n))\|\\
&\leq (1-\lambda)\|(d_1(x_1,z_1),\ldots,d_n(x_n,z_n))\| + \lambda\|(d_1(y_1,z_1),\ldots,d_n(y_n,z_n))\| \\
&= (1-\lambda)\left(\sum_{i=1}^n d_i^2(x_i,z_i) \right)^\frac12 + \lambda\left(\sum_{i=1}^n d_i^2(y_i,z_i) \right)^\frac12 \\
&= (1-\lambda)d(x,z) + \lambda d(y,z).
\end{align*}

We prove $(W2)$. Let $x=(x_1,\ldots,x_n)$, $y=(y_1,\ldots,y_n) \in X$ and $\lambda$, $\lambda' \in [0,1]$. We have that
\begin{align*}
d(W(x,y,\lambda),W(x,y,\lambda')) &=  \left(\sum_{i=1}^n d_i^2(W(x,y,\lambda)_i,W(x,y,\lambda')_i) \right)^\frac12 \\
&= \left(\sum_{i=1}^n d_i^2(W_i(x_i,y_i,\lambda),W_i(x_i,y_i,\lambda')) \right)^\frac12 \\
&= \left(\sum_{i=1}^n |\lambda-\lambda'|^2d_i^2(x_i,y_i) \right)^\frac12 \\
&= |\lambda-\lambda'|\left(\sum_{i=1}^n d_i^2(x_i,y_i) \right)^\frac12 \\
&= |\lambda-\lambda'|d(x,y).
\end{align*}

$(W3)$ is immediate. We now prove $(W4)$. Let $x=(x_1,\ldots,x_n)$, $y=(y_1,\ldots,y_n)$, $z=(z_1,\ldots,z_n)$, $w=(w_1,\ldots,w_n) \in X$ and $\lambda \in [0,1]$. We have that
\begin{align*}
d(W(x,y,\lambda),W(z,w,\lambda)) &= \left(\sum_{i=1}^n d_i^2(W(x,y,\lambda)_i,W(z,q,\lambda)_i) \right)^\frac12 \\
&= \left(\sum_{i=1}^n d_i^2(W_i(x_i,y_i,\lambda),W_i(z_i,w_i,\lambda)) \right)^\frac12 \\
&\leq \left(\sum_{i=1}^n ((1-\lambda)d_i(x_i,z_i) + \lambda d_i(y_i,w_i))^2 \right)^\frac12 \\
&=\|((1-\lambda)d_1(x_1,z_1) + \lambda d_1(y_1,w_1),\ldots,(1-\lambda)d_n(x_n,z_n) + \lambda d_n(y_n,w_n))\| \\
&=\|(1-\lambda)(d_1(x_1,z_1),\ldots,d_n(x_n,z_n)) + \lambda(d_1(y_1,w_1),\ldots,d_n(y_n,w_n))\|\\
&\leq (1-\lambda)\|(d_1(x_1,z_1),\ldots,d_n(x_n,z_n))\| + \lambda\|(d_1(y_1,w_1),\ldots,d_n(y_n,w_n))\| \\
&= (1-\lambda)\left(\sum_{i=1}^n d_i^2(x_i,z_i) \right)^\frac12 + \lambda\left(\sum_{i=1}^n d_i^2(y_i,w_i) \right)^\frac12 \\
&= (1-\lambda)d(x,z) + \lambda d(y,w).
\end{align*}
\end{proof}

\begin{proposition}
Assume that, for each $i$, $(X_i,d_i,W_i)$ is a CAT(0) space. Then $(X,d,W)$ is a CAT(0) space.
\end{proposition}

\begin{proof}
We shall use \eqref{cat0}. Let $a=(a_1,\ldots,a_n)$, $x=(x_1,\ldots,x_n)$, $y=(y_1,\ldots,y_n) \in X$. We have that
\begin{align*}
d^2\left(\frac{x+y}2,a\right) &= \sum_{i=1}^n d_i^2\left(\frac{x_i+y_i}2,a_i\right) \\
&\leq \sum_{i=1}^n \left(\frac12 d_i^2(x_i,a_i) + \frac12 d_i^2(y_i,a_i) - \frac14 d_i^2(x_i,y_i)\right)\\
&= \frac12 d^2(x,a) + \frac12 d^2(y,a) - \frac14 d^2(x,y),
\end{align*}
and we are done.
\end{proof}

The proof of the following proposition uses ideas\footnote{In addition, the proof benefited much from the quite perspicuous presentation in Ulrich Kohlenbach's private notes, which he shared with us and which included the computation of the modulus in the linear setting, a particular case of which being mentioned in the Introduction of the paper \cite{FreKoh23}.} from \cite[Theorem 1]{Cla36}.

\begin{proposition}\label{ucwpres}
Assume that, for each $i$, $(X_i,d_i,W_i)$ is a $UCW$-hyperbolic space. Then $(X,d,W)$ is a $UCW$-hyperbolic space.
\end{proposition}

\begin{proof}
For each $i$, let $\eta_i$ be a monotone modulus of uniform convexity for $(X_i,d_i,W_i)$. Define $\eta$, $\check{\eta} :(0, \infty) \times (0, 2] \to (0,\infty)$ by putting, for any $r >0$ and any $\eps \in (0,2]$,
$$\eta(r,\eps):= \min \left(\left\{ \eta_i(r,\eps) \mid i \in \{1,\ldots,n\}\right\} \cup \left\{\frac12\right\}\right)$$
and
$$\check{\eta}(r,\eps) := \min \left( \frac{\eps^4}{4608n^4} \cdot \eta^2\left(r,\frac\eps{\sqrt{n}}\right), \frac{\eps^2}{16n} \cdot \eta\left(r,\frac\eps{\sqrt{n}}\right) \right).$$
Clearly, $\eta$ and $\check{\eta}$ are monotone and we have that, for any $r>0$ and any $\eps \in (0,2]$, $\eta(r,\eps) \leq 1/2$ and, for any $i\in \{1,\ldots,n\}$, $\eta(r,\eps) \leq \eta_i(r,\eps)$. We will show that $\check{\eta}$ is a modulus of uniform convexity for $(X,d,W)$.

Let $r>0$, $\eps \in (0,2]$ and $a=(a_1,\ldots,a_n)$, $x=(x_1,\ldots,x_n)$, $y=(y_1,\ldots,y_n) \in X$ with $d(x,a) \leq r$, $d(y,a) \leq r$, $d(x,y) \geq \eps r$. It is immediate that, for each $i \in \{1,\ldots,n\}$, $d_i(x_i,a_i) \leq r$ and $d_i(y_i,a_i) \leq r$.

Set
$$\delta := \frac{\eps^2r}{24n^2}\eta\left(r,\frac{\eps}{\sqrt{n}}\right) \leq \frac{4r}{24n^2} \cdot \frac12 \leq r.$$
We will distinguish two cases.

{\bf Case I.} There is a $j \in \{1,\ldots, n\}$ such that $d_j(x_j,a_j) \geq d_j(y_j,a_j) + \delta$.

Note that, in this case, $\delta \leq d_j(y_j,a_j) + \delta \leq d_j(x_j,a_j) \leq r$, so $\delta^2/8r \leq r$.

Put $u:=(d_1(x_1,a_1),\ldots,d_n(x_n,a_n))$ and $v:=(d_1(y_1,a_1),\ldots,d_n(y_n,a_n))$, so $\|u\| = d(x,a) \leq r$ and $\|v\| = d(y,a) \leq r$. We have that
$$\|u-v\|^2 = \sum_{i=1}^n (d_i(x_i,a_i) - d_i(y_i,a_i))^2 \geq (d_j(x_j,a_j)-d_j(y_j,a_j))^2 \geq \delta^2,$$
so 
\begin{align*}
d^2\left(\frac{x+y}2,a\right) &= \sum_{i=1}^n d^2\left(\frac{x_i+y_i}2,a_i \right) \\
&\leq \sum_{i=1}^n \left(\frac{d(x_i,a_i)+d(y_i,a_i)}2\right)^2 \\
&= \left\|\frac{u+v}2\right\|^2 \\
&= \frac12 \|u\|^2 + \frac12 \|v\|^2 - \frac14 \|u-v\|^2 \\
&\leq r^2 - \frac14\delta^2 \leq r^2 - \frac14\delta^2 + \frac{\delta^4}{64r^2} = \left(r-\frac{\delta^2}{8r}\right)^2,
\end{align*}
so
\begin{align*}
d\left(\frac{x+y}2,a\right) &\leq \left|r-\frac{\delta^2}{8r}\right| = r-\frac{\delta^2}{8r} = \left(1-\frac{\delta^2}{8r^2}\right)r \\
&= \left(1 - \frac{\eps^4}{4608 n^4} \eta^2\left(r,\frac{\eps}{\sqrt{n}}\right)\right) r \leq (1-\check{\eta}(r,\eps))r,
\end{align*}
and we are done in this case.

{\bf Case II.} For any $i \in \{1,\ldots, n\}$, $d_i(x_i,a_i) \leq d_i(y_i,a_i) + \delta$.

Assume towards a contradiction that, for every $i \in \{1,\ldots,n\}$, $d_i(x_i,y_i) < \eps r/\sqrt{n}$. Then
$$\eps^2r^2 \leq d^2(x,y) = \sum_{i=1}^n d_i^2(x_i,y_i) < n \cdot \frac{\eps^2r^2}n = \eps^2r^2,$$
a contradiction. Thus, there is a $j \in \{1,\ldots,n\}$ such that $ d_j(x_j,y_j) \geq \eps r/\sqrt{n}$.

Set $K:= \max(d_j(x_j,a_j),d_j(y_j,a_j)) \leq r$, so
$$\frac{\eps r}{2\sqrt{n}} \leq K \leq d_j(y_j,a_j) + \delta.$$
Also, we have that $d_j(x_j,a_j) \leq K$, $d_j(y_j,a_j) \leq K$ and
$$d_j(x_j,y_j) \geq \frac{\eps r}{\sqrt{n}} \geq \frac\eps{\sqrt{n}} \cdot K,$$
so
\begin{align*}
d_j\left(\frac{x_j+y_j}2,a_j\right) &\leq \left(1- \eta_j\left(K, \frac\eps{\sqrt{n}} \right) \right) \cdot K \\
&\leq \left(1- \eta_j\left(r, \frac\eps{\sqrt{n}} \right) \right) \cdot K \\
&\leq \left(1- \eta\left(r, \frac\eps{\sqrt{n}} \right) \right) \cdot K \\
&\leq d_j(y_j,a_j) + \delta - \frac{\eps r}{2\sqrt{n}} \eta\left(r, \frac\eps{\sqrt{n}} \right).
\end{align*}
Therefore,
\begin{align*}
d^2\left(\frac{x+y}2,a\right) &= d_j^2\left(\frac{x_j+y_j}2,a_j\right) + \sum_{i \neq j} d_i^2\left(\frac{x_i+y_i}2,a_i\right) \\
&\leq \left(d_j(y_j,a_j) + \delta - \frac{\eps r}{2 \sqrt{n}} \eta\left(r, \frac{\eps}{\sqrt{n}}\right)\right)^2 + \sum_{i \neq j} \left(d_i(y_i,a_i) + \frac\delta 2\right)^2.
\end{align*}
Now put
$$A:=d_j(y_j,a_j) + \delta,$$
so
$$2A \geq \frac{\eps r}{\sqrt{n}},$$
and
$$B:=\frac{\eps r}{2 \sqrt{n}} \eta\left(r, \frac{\eps}{\sqrt{n}}\right) \leq  \frac{\eps r}{4\sqrt{n}} \leq  \frac{\eps r}{2\sqrt{n}},$$
so
$$ 2A - B \geq   \frac{\eps r}{2\sqrt{n}}$$
and
$$(A-B)^2 - A^2 = -B(2A-B) \leq -B \cdot  \frac{\eps r}{2\sqrt{n}} = - \frac{\eps^2 r^2}{4n} \eta\left(r, \frac{\eps}{\sqrt{n}}\right).$$
Thus,
\begin{align*}
d^2\left(\frac{x+y}2,a\right) &\leq \sum_{i=1}^n (d_i(y_i,a_i) + \delta)^2 - \frac{\eps^2 r^2}{4n} \eta\left(r, \frac{\eps}{\sqrt{n}}\right) \\
&= \sum_{i=1}^n (d_i^2(y_i,a_i) + \delta(\delta+2d_i(y_i,a_i))) - \frac{\eps^2 r^2}{4n} \eta\left(r, \frac{\eps}{\sqrt{n}}\right) \\
&\leq \sum_{i=1}^n (d_i^2(y_i,a_i) + 3r\delta) - \frac{\eps^2 r^2}{4n} \eta\left(r, \frac{\eps}{\sqrt{n}}\right) \\
&= \sum_{i=1}^n d_i^2(y_i,a_i) + 3rn\delta - \frac{\eps^2 r^2}{4n} \eta\left(r, \frac{\eps}{\sqrt{n}}\right) \\
&= d^2(y,a)  - \frac{\eps^2 r^2}{8n} \eta\left(r, \frac{\eps}{\sqrt{n}}\right).
\end{align*}
Now put $C:=d^2(y,a) \leq r^2$, so $\sqrt{C}\leq r$, and
$$D:= \frac{\eps^2 r^2}{8n} \eta\left(r, \frac{\eps}{\sqrt{n}}\right).$$
Clearly, from the inequality from before, $C \geq D$, so
$$\sqrt{C} \geq \frac{D}{2\sqrt{C}}.$$
We have that
$$d^2\left(\frac{x+y}2,a\right) \leq C-D \leq C-D+\frac{D^2}{4C} = \left(\sqrt{C} - \frac{D}{2\sqrt{C}} \right)^2,$$
so
\begin{align*}
d\left(\frac{x+y}2,a\right) &\leq \left|\sqrt{C} - \frac{D}{2\sqrt{C}}\right| = \sqrt{C} - \frac{D}{2\sqrt{C}} \leq r - \frac{D}{2r} = r -  \frac{\eps^2 r}{16n} \eta\left(r, \frac{\eps}{\sqrt{n}}\right) \\
&= \left(1 - \frac{\eps^2 }{16n} \eta\left(r, \frac{\eps}{\sqrt{n}}\right)\right)r \leq (1-\check{\eta}(r,\eps))r,
\end{align*}
so we are done in this case, too.
\end{proof}

We remark that, if we only care about the property $(G)$, the corresponding preservation proof is considerably simpler.

\begin{proposition}
Assume that, for each $i$, $(X_i,d_i,W_i)$ has property $(G)$. Then $(X,d,W)$ has property $(G)$.
\end{proposition}

\begin{proof}
For each $i$, let $\psi_i$ be a modulus for the property $(G)$ of $(X_i,d_i,W_i)$. Define $\psi :(0, \infty) \times (0, \infty) \to (0,\infty)$ by putting, for any $r$, $\eps >0$,
$$\psi(r,\eps) := \min_i\ \psi_i \left( r, \frac{\eps}{\sqrt{n}} \right).$$
We will show that $(X,d,W)$ has property $(G)$ with modulus $\psi$.

Let $r$, $\eps>0$ and $a=(a_1,\ldots,a_n)$, $x=(x_1,\ldots,x_n)$, $y=(y_1,\ldots,y_n) \in X$ with $d(x,a) \leq r$, $d(y,a) \leq r$, $d(x,y) \geq \eps $. It is immediate that, for each $i \in \{1,\ldots,n\}$, $d_i(x_i,a_i) \leq r$ and $d_i(y_i,a_i) \leq r$. Assume towards a contradiction that, for every $i \in \{1,\ldots,n\}$, $d_i(x_i,y_i) < \eps/\sqrt{n}$. Then
$$\eps^2 \leq d^2(x,y) = \sum_{i=1}^n d_i^2(x_i,y_i) < n \cdot \frac{\eps^2}n = \eps^2,$$
a contradiction. Thus, there is a $j \in \{1,\ldots,n\}$ such that $ d_j(x_j,y_j) \geq \eps/\sqrt{n}$.

We have, using Proposition~\ref{quad} for the first inequality, that
\begin{align*}
\frac12 d^2(x,a) + \frac12 d^2(y,a) - d^2\left(\frac{x+y}2,a\right) &= \sum_{i=1}^n \left( \frac{d_i^2(x_i,a_i) + d_i^2(y_i,a_i)}2 - d_i^2\left(\frac{x_i+y_i}2,a_i\right) \right) \\
&\geq  \frac{d_j^2(x_j,a_j) + d_j^2(y_j,a_j)}2 - d_j^2\left(\frac{x_j+y_j}2,a_j\right)  \\
&\geq  \psi_j \left( r, \frac{\eps}{\sqrt{n}} \right) \\
& \geq  \psi(r,\eps),
\end{align*}
which is what we wanted to show.
\end{proof}

The proof of the preservation of the property $(M)$ follows the general structure of the proof of Proposition~\ref{ucwpres}.

\begin{proposition}
Assume that, for each $i$, $(X_i,d_i,W_i)$ has property $(M)$. Then $(X,d,W)$ has property $(M)$.
\end{proposition}

\begin{proof}
For each $i$, let $\psi_i$ be a modulus for the property $(M)$ of $(X_i,d_i,W_i)$. Define $\psi$, $\check{\psi} :(0, \infty) \times (0, \infty) \to (0,\infty)$ by putting, for any $r$, $\eps >0$,
$$\psi(r,\eps) := \min_i\ \psi_i(r,\eps)$$
and 
$$\check{\psi}(r,\eps):= \min\left(\frac{\psi^2(r,\eps/\sqrt{n})}{64n^2r^2},\frac{\psi(r,\eps/\sqrt{n})}2\right).$$
We will show that $(X,d,W)$ has property $(M)$ with modulus $\check{\psi}$.

Let $r$, $\eps>0$ and $a=(a_1,\ldots,a_n)$, $x=(x_1,\ldots,x_n)$, $y=(y_1,\ldots,y_n) \in X$ with $d(x,a) \leq r$, $d(y,a) \leq r$, $d(x,y) \geq \eps $. It is immediate that, for each $i \in \{1,\ldots,n\}$, $d_i(x_i,a_i) \leq r$ and $d_i(y_i,a_i) \leq r$.

Set
$$\delta := \frac{\psi(r,\eps/\sqrt{n})}{2n}.$$
We will distinguish two cases.

{\bf Case I.} There is a $j \in \{1,\ldots, n\}$ such that $d_j^2(x_j,a_j) \geq d_j^2(y_j,a_j) + \delta$.

Note that, in this case, $d_j(x_j,a_j) \geq d_j(y_j,a_j)$, so
$$\delta \leq d_j^2(x_j,a_j) - d_j^2(y_j,a_j) \leq (d_j(x_j,a_j) - d_j(y_j,a_j)) \cdot 2r,$$
from which we get that $d_j(x_j,a_j) - d_j(y_j,a_j) \geq \delta/2r$.

Put $u:=(d_1(x_1,a_1),\ldots,d_n(x_n,a_n))$ and $v:=(d_1(y_1,a_1),\ldots,d_n(y_n,a_n))$, so $\|u\| = d(x,a) \leq r$ and $\|v\| = d(y,a) \leq r$. We have that
$$\|u-v\|^2 = \sum_{i=1}^n (d_i(x_i,a_i) - d_i(y_i,a_i))^2 \geq (d_j(x_j,a_j)-d_j(y_j,a_j))^2 \geq \frac{\delta^2}{4r^2},$$
so 
\begin{align*}
d^2\left(\frac{x+y}2,a\right) &= \sum_{i=1}^n d^2\left(\frac{x_i+y_i}2,a_i \right) \\
&\leq \sum_{i=1}^n \left(\frac{d(x_i,a_i)+d(y_i,a_i)}2\right)^2 \\
&= \left\|\frac{u+v}2\right\|^2 \\
&= \frac12 \|u\|^2 + \frac12 \|v\|^2 - \frac14 \|u-v\|^2 \\
&\leq \max(\|u\|^2, \|v\|^2) - \frac{\delta^2}{16r^2} \\
&= \max(d^2(x,a),d^2(y,a)) - \frac{\psi^2(r,\eps/\sqrt{n})}{64n^2r^2} \\
&\leq \max(d^2(x,a),d^2(y,a)) - \check{\psi}(r,\eps),
\end{align*}
so we are done in this case.

{\bf Case II.} For any $i \in \{1,\ldots, n\}$, $d_i^2(x_i,a_i) \leq d_i^2(y_i,a_i) + \delta$.

Assume towards a contradiction that, for every $i \in \{1,\ldots,n\}$, $d_i(x_i,y_i) < \eps /\sqrt{n}$. Then
$$\eps^2 \leq d^2(x,y) = \sum_{i=1}^n d_i^2(x_i,y_i) < n \cdot \frac{\eps^2}n = \eps^2,$$
a contradiction. Thus, there is a $j \in \{1,\ldots,n\}$ such that $ d_j(x_j,y_j) \geq \eps /\sqrt{n}$.

Therefore, using Proposition~\ref{quad} for the first inequality,
\begin{align*}
d^2\left(\frac{x+y}2,a\right) &= \sum_{i \neq j} d_i^2\left(\frac{x_i+y_i}2,a_i\right) + d_j^2\left(\frac{x_j+y_j}2,a_j\right)   \\
&\leq \sum_{i \neq j} \max(d_i^2(x_i,a_i),d_i^2(y_i,a_i)) + \max(d_j^2(x_j,a_j),d_j^2(y_j,a_j))- \psi_j(r,\eps/\sqrt{n}) \\
&\leq \sum_{i=1}^n (d_i^2(y_i,a_i) + \delta)- \psi(r,\eps/\sqrt{n})\\
&= d^2(y,a) + n\delta - \psi(r,\eps/\sqrt{n}) \\
&= d^2(y,a) - \frac{\psi(r,\eps/\sqrt{n})}2 \\
&\leq \max(d^2(x,a),d^2(y,a)) - \frac{\psi(r,\eps/\sqrt{n})}2 \\
&\leq \max(d^2(x,a),d^2(y,a)) - \check{\psi}(r,\eps),
\end{align*}
so we are done in this case, too.
\end{proof}

Now fix $\pi_1,\ldots,\pi_n$ such that, for each $i$, $(X_i,d_i,W_i,\pi_i)$ is a smooth hyperbolic space. Define $\pi: X^2 \times X^2 \to \R$ by putting, for any $x=(x_1,\ldots,x_n)$, $y=(y_1,\ldots,y_n)$, $u=(u_1,\ldots,u_n)$, $v=(v_1,\ldots,v_n) \in X$,
$$\piv{xy}{uv} := \sum_{i=1}^n \piiv{i}{x_iy_i}{u_iv_i}.$$

\begin{proposition}
$(X,d,W,\pi)$ is a smooth hyperbolic space.
\end{proposition}

\begin{proof}
$(P1)$, $(P2)$ and $(P3)$ are immediate. We now prove $(P4)$. Let $x=(x_1,\ldots,x_n)$, $y=(y_1,\ldots,y_n)$, $u=(u_1,\ldots,u_n)$, $v=(v_1,\ldots,v_n) \in X$. We have that
\begin{align*}
\piv{xy}{uv} &= \sum_{i=1}^n \piiv{i}{x_iy_i}{u_iv_i}\\
&\leq \sum_{i=1}^n d_i(x_i,y_i)d_i(u_i,v_i)\\
&= \langle (d_1(x_1,y_1),\ldots,d_n(x_n,y_n)), (d_1(u_1,v_1),\ldots,d_n(u_n,v_n)) \rangle \\
&\leq \|(d_1(x_1,y_1),\ldots,d_n(x_n,y_n))\| \cdot \|(d_1(u_1,v_1),\ldots,d_n(u_n,v_n))\| \\
&= \left(\sum_{i=1}^n d_i^2(x_i,y_i) \right)^\frac12 \cdot\left(\sum_{i=1}^n d_i^2(u_i,v_i) \right)^\frac12 \\
&= d(x,y) d(u,v).
\end{align*}

We prove $(P5)$. Let $x=(x_1,\ldots,x_n)$, $y=(y_1,\ldots,y_n)$, $z=(z_1,\ldots,z_n) \in X$ and $\lambda \in [0,1]$. We have that
\begin{align*}
d^2(W(x,y,\lambda),z) &= \sum_{i=1}^n d_i^2(W(x,y,\lambda)_i,z_i)  \\
&= \sum_{i=1}^n d_i^2(W_i(x_i,y_i,\lambda),z_i) \\
&\leq \sum_{i=1}^n \left((1-\lambda)^2d_i^2(x_i,z_i) + 2\lambda\piiv{i}{y_iz_i}{W_i(x_i,y_i,\lambda)z_i} \right)  \\
&= \sum_{i=1}^n \left((1-\lambda)^2d_i^2(x_i,z_i) + 2\lambda\piiv{i}{y_iz_i}{W(x,y,\lambda)_iz_i} \right)  \\
&= (1-\lambda)^2 d^2(x,z) + 2\lambda\piv{yz}{W(x,y,\lambda)z}.
\end{align*}
\end{proof}

\begin{proposition}
Assume that, for each $i$, $(X_i,d_i,W_i,\pi_i)$ is a uniformly smooth hyperbolic space. Then $(X,d,W,\pi)$ is a uniformly smooth hyperbolic space.
\end{proposition}

\begin{proof}
For each $i$, let $\omega_i$ be a modulus of uniform continuity for $\pi_i$. Define $\omega :(0, \infty) \times (0, \infty) \to (0,\infty)$ by putting, for any $r$, $\eps>0$,
$$\omega(r,\eps) := \min_i\ \omega_i \left( r, \frac{\eps}{\sqrt{n}} \right).$$
We will show that $\omega$ is a modulus of uniform continuity for $\pi$.

Let  $r$, $\eps>0$ and $a = (a_1,\ldots,a_n)$, $u = (u_1,\ldots,u_n)$, $v = (v_1,\ldots,v_n)$, $x = (x_1,\ldots,x_n)$, $y = (y_1,\ldots,y_n) \in X$ with $d(u,a) \leq r$, $d(v,a) \leq r$ and $d(u,v) \leq \omega(r,\eps)$. We want to show that
$$| \piv{xy}{ua} - \piv{xy}{va}| \leq \eps \cdot d(x,y).$$

It is immediate that, for each $i$, $d_i(u_i,a_i) \leq r$, $d_i(v_i,a_i) \leq r$ and
$$d_i(u_i,v_i) \leq \omega(r,\eps) \leq \omega_i \left( r, \frac{\eps}{\sqrt{n}} \right).$$
Thus, for each $i$,
$$| \piiv{i}{x_iy_i}{u_ia_i} - \piiv{i}{x_iy_i}{v_ia_i}| \leq \frac{\eps}{\sqrt{n}} \cdot d_i(x_i,y_i).$$

We have that
\begin{align*}
| \piv{xy}{ua} - \piv{xy}{va}| &= \left| \sum_{i=1}^n \left(\piiv{i}{x_iy_i}{u_ia_i} - \piiv{i}{x_iy_i}{v_ia_i}\right)\right| \\
&\leq \sum_{i=1}^n\left|\piiv{i}{x_iy_i}{u_ia_i} - \piiv{i}{x_iy_i}{v_ia_i}\right| \\
&\leq \frac{\eps}{\sqrt{n}} \sum_{i=1}^n d_i(x_i,y_i)\\
&= \frac{\eps}{\sqrt{n}} \cdot \langle (1,\ldots,1) , (d_1(x_1,y_1),\ldots,d_n(x_n,y_n)) \rangle\\
&\leq \frac{\eps}{\sqrt{n}} \cdot \|(1,\ldots,1)\| \cdot \|(d_1(x_1,y_1),\ldots,d_n(x_n,y_n)) \| \\
&= \frac{\eps}{\sqrt{n}} \cdot \sqrt{n} \cdot \left(\sum_{i=1}^n d_i^2(x_i,y_i) \right)^\frac12\\
&= \eps \cdot d(x,y),
\end{align*}
and we are done.
\end{proof}

\section{The example}\label{sec:xmpl}

We now proceed to give the announced example. Denote by $\R^2_3$ the Euclidean plane equipped with the $3$-norm. It is known that this is a uniformly convex and uniformly smooth Banach space which is not a Hilbert space. From the above results, it follows that $\Hh \times \R^2_3$ is a complete uniformly smooth $UCW$-hyperbolic space, and, thus, Reich's theorem and its consequences hold for it. We will now show that this space is neither a CAT(0) space, nor a convex subset of a normed space (with the canonical convexity structure).

\begin{proposition}
$\R^2_3$ is not a CAT(0) space.
\end{proposition}

\begin{proof}
It is known that a Banach space is a CAT(0) space iff it is Hilbert.
\end{proof}

\begin{corollary}
$\Hh \times \R^2_3$ is not a CAT(0) space.
\end{corollary}

\begin{proof}
We use the fact that $\R^2_3$ can be seen as a convex subset of $\Hh \times \R^2_3$.
\end{proof}

\begin{proposition}
$\Hh$ is not a convex subset of a normed space.
\end{proposition}

\begin{proof}
It is immediate that any $W$-hyperbolic space $X$ arising from the canonical convexity structure of a convex subset of a normed space satisfies the following: for any $x$, $y$, $z \in X$ and any $\lambda \in [0,1]$,
$$d(W(x,z,\lambda),W(y,z,\lambda)) = (1-\lambda)d(x,y).$$
We will now exhibit $x$, $y$, $z \in \Hh$ and $\lambda \in [0,1]$ such that
$$d(W(x,z,\lambda),W(y,z,\lambda)) < (1-\lambda)d(x,y).$$
Take $x:=(0,1)$, $y:=(1,1)$, $z:=(0,2)$, $\lambda:=1/2$. By \eqref{mid-h}, we have that $W(x,z,\lambda)=(0,\sqrt{2})$ and $W(y,z,\lambda)=(2/3,2\sqrt{5}/3)$. Using \eqref{dist-h}, we get that
$$d(x,y) = \arcosh(3/2) = \ln((3+\sqrt{5})/2)$$
and
$$d(W(x,z,\lambda),W(y,z,\lambda)) = \arcosh (7/(2\sqrt{10})) = \ln(\sqrt{10}/2).$$

We have to show that
$$\ln(\sqrt{10}/2) < \ln((3+\sqrt{5})/2)/2.$$
But this follows from the strict inequality
$$(\sqrt{10}/2)^2 < (3+\sqrt{5})/2,$$
which is immediate.
\end{proof}

\begin{corollary}
$\Hh \times \R^2_3$ is not a convex subset of a normed space.
\end{corollary}

\begin{proof}
We use the fact that $\Hh$ can be seen as a convex subset of $\Hh \times \R^2_3$.
\end{proof}

\section{Acknowledgements}

The first author was supported by the German Science Foundation (DFG) project Ref.\ PI 2070/1-1, Projektnummer 549333475.

The authors would like to thank Ulrich Kohlenbach for his input regarding the preservation of uniform convexity under products.

\end{document}